\RequirePackage{snapshot}
\documentclass[oneside,10pt,]{article}
\usepackage{geometry}
\usepackage{ifthen}
\usepackage{etoolbox}
\usepackage{ifxetex,ifluatex}
\usepackage{graphicx}
\PassOptionsToPackage{usenames,dvipsnames,svgnames,table}{xcolor}
\usepackage{xcolor}
\usepackage{tcolorbox}
\tcbuselibrary{skins}
\tcbuselibrary{breakable}
\tcbuselibrary{raster}
\tcbset{ bwminimalstyle/.style={size=minimal, boxrule=-0.3pt, frame empty,
colback=white, colbacktitle=white, coltitle=black, opacityfill=0.0} }
\tcbset{ runintitlestyle/.style={fonttitle=\blocktitlefont\upshape\bfseries, attach title to upper} }
\tcbset{ exercisespacingstyle/.style={before skip={1.5ex plus 0.5ex}} }
\tcbset{ blockspacingstyle/.style={before skip={2.0ex plus 0.5ex}} }
\tcbuselibrary{xparse}
\IfFileExists{latexrelease.sty}{}{\usepackage{fixltx2e}}
\usepackage{xstring}
\ifthenelse{\boolean{xetex} \or \boolean{luatex}}{%
\ifxetex\usepackage{xltxtra}\fi
\ifluatex\usepackage{realscripts}\fi
}{
\usepackage[utf8]{inputenc}
}
\newcommand{\divisionfont}{\relax}
\newcommand{\blocktitlefont}{\relax}
\newcommand{\contentsfont}{\relax}

\newcommand{\xreffont}{\relax}

\ifthenelse{\boolean{xetex} \or \boolean{luatex}}{%
\usepackage{fontspec}
\IfFontExistsTF{Inconsolatazi4-Regular.otf}{}{\GenericError{}{The font "Inconsolatazi4-Regular.otf" requested by PreTeXt output is not available.  Either a file cannot be located in default locations via a filename, or a font is not known by its name as part of your system.}{Consult the PreTeXt Guide for help with LaTeX fonts.}{}}
\IfFontExistsTF{Inconsolatazi4-Bold.otf}{}{\GenericError{}{The font "Inconsolatazi4-Bold.otf" requested by PreTeXt output is not available.  Either a file cannot be located in default locations via a filename, or a font is not known by its name as part of your system.}{Consult the PreTeXt Guide for help with LaTeX fonts.}{}}
\usepackage{zi4}
\setmonofont[BoldFont=Inconsolatazi4-Bold.otf,StylisticSet={1,3}]{Inconsolatazi4-Regular.otf}
\usepackage{polyglossia}
\setmainlanguage[variant=usmax]{english}
}{%
\usepackage{lmodern}
\usepackage[T1]{fontenc}
\usepackage[varqu,varl]{inconsolata}
}
\usepackage{microtype}
\usepackage{amsmath}
\usepackage{amscd}
\usepackage{amssymb}
\allowdisplaybreaks[4]
\usepackage[explicit, newparttoc, pagestyles]{titlesec}
\usepackage{titletoc}
\newcommand{\divisionnameptx}{\relax}%
\newcommand{\authorsptx}{\relax}%
\NewDocumentEnvironment{sectionptx}{mmmmmm}
{%
\renewcommand{\divisionnameptx}{Section}%
\renewcommand{\titleptx}{#1}%
\renewcommand{\subtitleptx}{#2}%
\renewcommand{\shortitleptx}{#3}%
\renewcommand{\authorsptx}{#4}%
\renewcommand{\epigraphptx}{#5}%
\section[{#3}]{#1}%
\label{#6}%
}{}%
\NewDocumentEnvironment{references-section}{mmmmmm}
{%
\renewcommand{\divisionnameptx}{References}%
\renewcommand{\titleptx}{#1}%
\renewcommand{\subtitleptx}{#2}%
\renewcommand{\shortitleptx}{#3}%
\renewcommand{\authorsptx}{#4}%
\renewcommand{\epigraphptx}{#5}%
\section[{#3}]{#1}%
\label{#6}%
}{}%
\NewDocumentEnvironment{references-section-numberless}{mmmmmm}
{%
\renewcommand{\divisionnameptx}{References}%
\renewcommand{\titleptx}{#1}%
\renewcommand{\subtitleptx}{#2}%
\renewcommand{\shortitleptx}{#3}%
\renewcommand{\authorsptx}{#4}%
\renewcommand{\epigraphptx}{#5}%
\section*{#1}%
\addcontentsline{toc}{section}{#3}
\label{#6}%
}{}%
\titleformat{\part}[display]
{\divisionfont\Huge\bfseries\centering}{\divisionnameptx\space\thepart}{30pt}{\Huge#1}
[{\Large\centering\authorsptx}]
\titleformat{\chapter}[display]
{\divisionfont\huge\bfseries}{\divisionnameptx\space\thechapter}{20pt}{\Huge#1}
[{\Large\authorsptx}]
\titleformat{name=\chapter,numberless}[display]
{\divisionfont\huge\bfseries}{}{0pt}{#1}
[{\Large\authorsptx}]
\titlespacing*{\chapter}{0pt}{50pt}{40pt}
\titleformat{\section}[hang]
{\divisionfont\Large\bfseries}{\thesection}{1ex}{#1}
[{\large\authorsptx}]
\titleformat{name=\section,numberless}[block]
{\divisionfont\Large\bfseries}{}{0pt}{#1}
[{\large\authorsptx}]
\titlespacing*{\section}{0pt}{3.5ex plus 1ex minus .2ex}{2.3ex plus .2ex}
\titleformat{\subsection}[hang]
{\divisionfont\large\bfseries}{\thesubsection}{1ex}{#1}
[{\normalsize\authorsptx}]
\titleformat{name=\subsection,numberless}[block]
{\divisionfont\large\bfseries}{}{0pt}{#1}
[{\normalsize\authorsptx}]
\titlespacing*{\subsection}{0pt}{3.25ex plus 1ex minus .2ex}{1.5ex plus .2ex}
\titleformat{\subsubsection}[hang]
{\divisionfont\normalsize\bfseries}{\thesubsubsection}{1em}{#1}
[{\small\authorsptx}]
\titleformat{name=\subsubsection,numberless}[block]
{\divisionfont\normalsize\bfseries}{}{0pt}{#1}
[{\normalsize\authorsptx}]
\titlespacing*{\subsubsection}{0pt}{3.25ex plus 1ex minus .2ex}{1.5ex plus .2ex}
\titleformat{\paragraph}[hang]
{\divisionfont\normalsize\bfseries}{\theparagraph}{1em}{#1}
[{\small\authorsptx}]
\titleformat{name=\paragraph,numberless}[block]
{\divisionfont\normalsize\bfseries}{}{0pt}{#1}
[{\normalsize\authorsptx}]
\titlespacing*{\paragraph}{0pt}{3.25ex plus 1ex minus .2ex}{1.5em}
\titlecontents{part}%
[0pt]{\contentsmargin{0em}\addvspace{1pc}\contentsfont\bfseries}%
{\Large\thecontentslabel\enspace}{\Large}%
{}%
[\addvspace{.5pc}]%
\titlecontents{chapter}%
[0pt]{\contentsmargin{0em}\addvspace{1pc}\contentsfont\bfseries}%
{\large\thecontentslabel\enspace}{\large}%
{\hfill\bfseries\thecontentspage}%
[\addvspace{.5pc}]%
\dottedcontents{section}[3.8em]{\contentsfont}{2.3em}{1pc}%
\dottedcontents{subsection}[6.1em]{\contentsfont}{3.2em}{1pc}%
\dottedcontents{subsubsection}[9.3em]{\contentsfont}{4.3em}{1pc}%
\newcommand{\mono}[1]{\texttt{#1}}
\newtcolorbox[auto counter, number within=section]{block}{}
\newtcolorbox[auto counter, number within=section]{project-distinct}{}
\makeatletter
\newtcolorbox[auto counter, number within=tcb@cnt@block, number freestyle={\noexpand\thetcb@cnt@block(\noexpand\alph{\tcbcounter})}]{subdisplay}{}
\makeatother
\tcbset{ theoremstyle/.style={bwminimalstyle, runintitlestyle, blockspacingstyle, after title={\space}, } }
\newtcolorbox[use counter from=block]{theorem}[3]{title={{Theorem~\thetcbcounter\notblank{#1#2}{\space}{}\notblank{#1}{\space#1}{}\notblank{#2}{\space(#2)}{}}}, phantomlabel={#3}, breakable, parbox=false, after={\par}, fontupper=\itshape, theoremstyle, }
\tcbset{ definitionstyle/.style={bwminimalstyle, runintitlestyle, blockspacingstyle, after title={\space}, after upper={\space\space\hspace*{\stretch{1}}\(\lozenge\)}, } }
\newtcolorbox[use counter from=block]{definition}[2]{title={{Definition~\thetcbcounter\notblank{#1}{\space\space#1}{}}}, phantomlabel={#2}, breakable, parbox=false, after={\par}, definitionstyle, }
\tcbset{ figureptxstyle/.style={bwminimalstyle, middle=1ex, blockspacingstyle, fontlower=\blocktitlefont} }
\newtcolorbox[use counter from=block]{figureptx}[3]{lower separated=false, before lower={{\textbf{Figure~\thetcbcounter}\space#1}}, phantomlabel={#2}, unbreakable, parbox=false, figureptxstyle, }
\NewDocumentEnvironment{introduction}{m}
{\notblank{#1}{\noindent\textbf{#1}\space}{}}{\par\medskip}
\tcbset{ proofstyle/.style={bwminimalstyle, fonttitle=\blocktitlefont\itshape, attach title to upper, after title={\space}, after upper={\space\space\hspace*{\stretch{1}}\(\blacksquare\)},
} }
\newtcolorbox{proof}[2]{title={\notblank{#1}{#1}{Proof.}}, phantom={\hypertarget{#2}{}}, breakable, parbox=false, after={\par}, proofstyle }

\numberwithin{equation}{section}
\tcbset{ imagestyle/.style={bwminimalstyle} }
\NewTColorBox{image}{mmm}{imagestyle,left skip=#1\linewidth,width=#2\linewidth}
\usepackage{fancyvrb}
\DefineVerbatimEnvironment{preformatted}{Verbatim}{}
\usepackage{enumitem}
\newlist{referencelist}{description}{4}
\setlist[referencelist]{leftmargin=!,labelwidth=!,labelsep=0ex,itemsep=1.0ex,topsep=1.0ex,partopsep=0pt,parsep=0pt}
\usepackage[hyperfootnotes=false]{hyperref}
\hypersetup{colorlinks=true,linkcolor=blue,citecolor=blue,filecolor=blue,urlcolor=blue}
\hypersetup{pdftitle={``Pass the Buck'' on a Rooted Tree}}
\providecommand\phantomsection{}
\tcbset{ sbsstyle/.style={raster before skip=2.0ex, raster equal height=rows, raster force size=false} }
\tcbset{ sbspanelstyle/.style={bwminimalstyle, fonttitle=\blocktitlefont} }
\NewDocumentEnvironment{sidebyside}{mmmm}
  {\begin{tcbraster}
    [sbsstyle,raster columns=#1,
    raster left skip=#2\linewidth,raster right skip=#3\linewidth,raster column skip=#4\linewidth]}
  {\end{tcbraster}}
\NewTColorBox{sbspanel}{mO{top}}{sbspanelstyle,width=#1\linewidth,valign=#2}
\usepackage{extpfeil}

\title{``Pass the Buck'' on a Rooted Tree}
\author{Kenneth Levasseur\\
Department of Mathematical Sciences\\
University of Massachusetts Lowell\\
Lowell, Massachusetts, USA\\
\href{mailto:kenneth_levasseur@uml.edu}{\nolinkurl{kenneth_levasseur@uml.edu}}
}
\date{January 6, 2021}
\begin{document}
\hypertarget{x:article:passthebuck}{}
\maketitle
\thispagestyle{empty}
\begin{abstract}
The Stochastic Abacus is can employed to compute winning probabilities for each vertex of a rooted tree in the game ``Pass the Buck'', with the starting vertex being the root of the tree. For all but the simplest trees, the abacus can't really be implemented due to the large number of steps needed for completion.  In this paper, a technique for anticipating the outcome is introduced.%
\end{abstract}
\begin{introduction}{Introduction.}%
In the 1970's, Engel \hyperlink{x:biblio:biblio-engel}{[{\xreffont 2}]} devised the Stochastic Abacus as a way to compute probabilities for certain discrete probability problems with minimal numerical computation.  More recently, Torrence \hyperlink{x:biblio:biblio-torrence}{[{\xreffont 7}]} used the same technique to determine winning probabilities for players in the game ``Pass the Buck'' for a variety of families of graphs.  The Stochastic Abacus has found more widespread exposure due to a recent article by Propp \hyperlink{x:biblio:biblio-propp}{[{\xreffont 5}]} in Math Horizons. The author initially applied the abacus to the game on a complete binary tree with the root as the starting vertex. In \hyperlink{x:biblio:biblio-levasseur}{[{\xreffont 4}]} this game was analyzed for complete binary trees, making use of the symmetry of these trees at all levels. One such tree is \hyperref[x:figure:fig-rooted-complete]{Figure~{\xreffont\ref{x:figure:fig-rooted-complete}}}. In this note, we use similar logic to describe how the game an arbitrary rooted tree can be analyzed, making it possible to anticipate the outcome for much more complex trees.  One such example is the rooted tree in \hyperref[x:figure:fig-rooted-random]{Figure~{\xreffont\ref{x:figure:fig-rooted-random}}}. Note that the roots of all rooted trees are drawn here with roots on the top.%
\begin{sidebyside}{2}{0}{0}{0}%
\begin{sbspanel}{0.5}%
\begin{figureptx}{A complete binary tree to level three}{x:figure:fig-rooted-complete}{}%
\includegraphics[width=\linewidth]{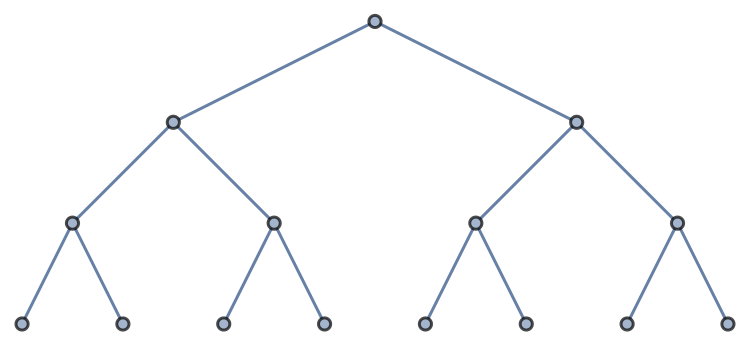}
\tcblower
\end{figureptx}%
\end{sbspanel}%
\begin{sbspanel}{0.5}%
\begin{figureptx}{A random rooted tree}{x:figure:fig-rooted-random}{}%
\includegraphics[width=\linewidth]{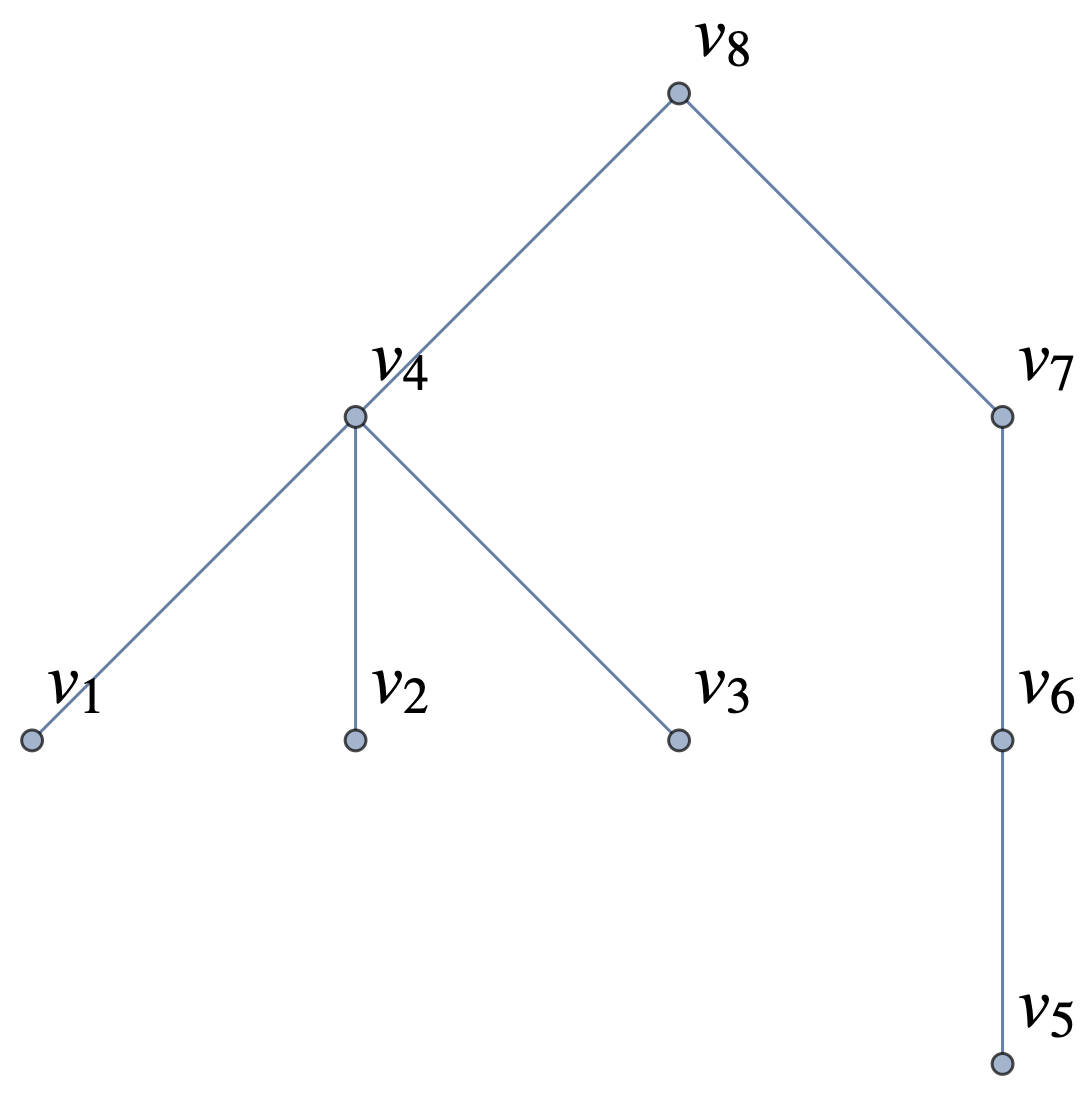}
\tcblower
\end{figureptx}%
\end{sbspanel}%
\end{sidebyside}%
\end{introduction}%
\typeout{************************************************}
\typeout{Section 1 Pass the Buck}
\typeout{************************************************}
\begin{sectionptx}{Pass the Buck}{}{Pass the Buck}{}{}{x:section:section-pass-the-buck}
The game ``Pass the Buck'' is played on a connected undirected graph, with a distinguished ``start vertex.'' The game proceeds in steps starting with the start vertex holding a prize (the ``buck'').  At every stage in the game, the current vertex that holds the buck and its neighboring vertices are selected randomly and uniformly. If the the current vertex is selected, the game ends with that vertex winning. If a neighboring vertex is selected the buck is passed there and process is repeated. More precisely, if the degree of the vertex that holds the buck is \(k\), then the buck moves to any of the neighbors with probability \(\frac{1}{k+1}\) and the game ends with the player at the current vertex winning with probability \(\frac{1}{k+1}\).%
\end{sectionptx}
\typeout{************************************************}
\typeout{Section 2 The Stochastic Abacus}
\typeout{************************************************}
\begin{sectionptx}{The Stochastic Abacus}{}{The Stochastic Abacus}{}{}{x:section:section-stochastic-abacus}
\begin{figureptx}{Final out come of the abacus on the random rooted tree}{x:figure:fig-abacus-random-tree}{}%
\begin{image}{0.2}{0.6}{0.2}%
\includegraphics[width=\linewidth]{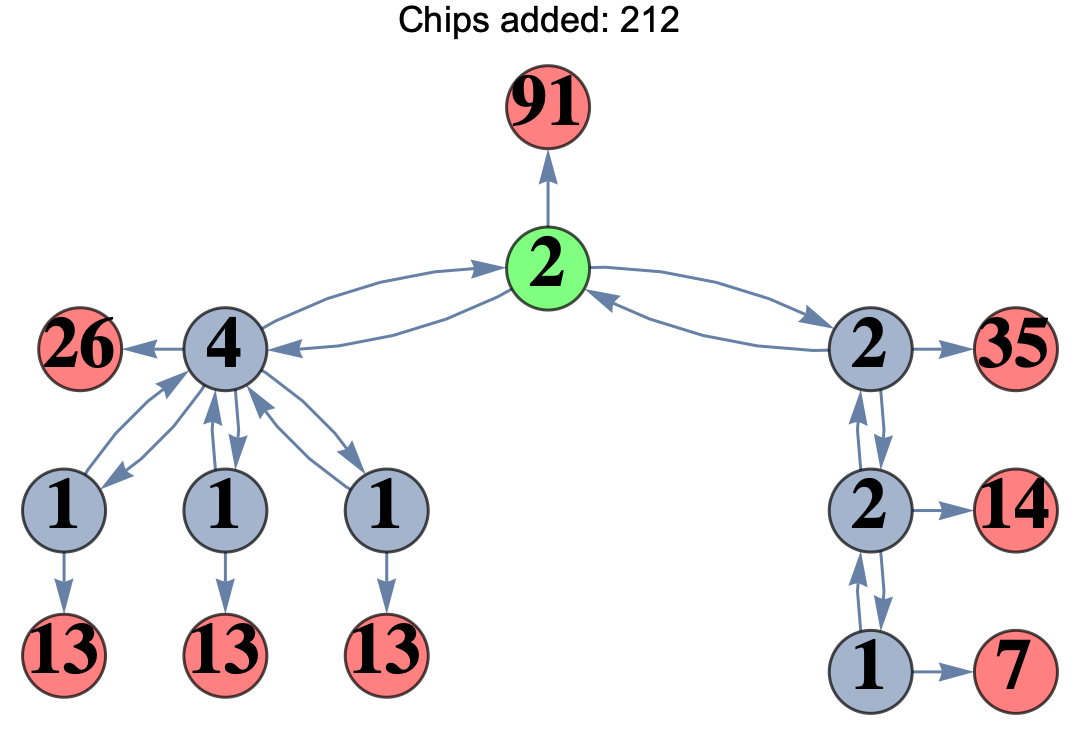}
\end{image}%
\tcblower
\end{figureptx}%
Engle's Stochastic Abacus is a chip firing algorithm that, when complete, allows computation of winning probabilities for Pass the Buck. The final outcome for the random rooted tree above is shown in the augmented directed graph, \hyperref[x:figure:fig-abacus-random-tree]{Figure~{\xreffont\ref{x:figure:fig-abacus-random-tree}}}.  The green vertex is the root of the tree, and the gray vertices are the other vertices of the tree. These vertices are labeled with the number of chips that are initially loaded into each vertex, one less than the outdegree of each of these vertices in the augmented graph you see here.  The pink vertices are absorbing vertices, one for each of the vertices in the tree. They accumulate chips in the implementation of the stochastic abacus. A total of 212 chips were added to the abacus after its initial critical loading, at which point the critial loading levels have been reached once more. This means that the root, whose chip count in its absorbing vertex is 91, has win probability \(\frac{91}{212}\). %
\par
A tree of this size is just about on the border of the sizes for which the abacus can reasonably completed manually. There are programs that can implement the abacus -  this is how the outcome above was actually computed - they are limited. Relatively simple trees with periods into the hundreds of thousands or more quickly put restrictions on this approach. For example, the stochastic abacus deposits over 64 million chips into the root of the tree in \hyperref[x:figure:fig-larger-tree]{Figure~{\xreffont\ref{x:figure:fig-larger-tree}}}.%
\begin{figureptx}{A slightly larger tree with high restoration number}{x:figure:fig-larger-tree}{}%
\begin{image}{0.2}{0.6}{0.2}%
\includegraphics[width=\linewidth]{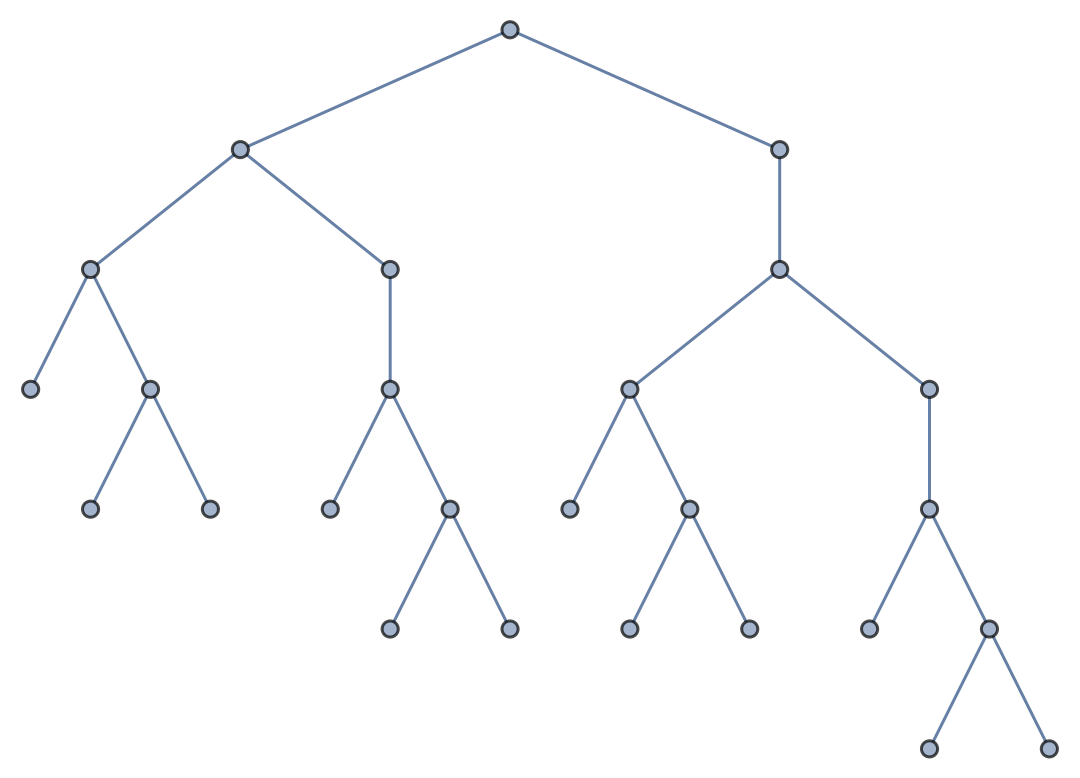}
\end{image}%
\tcblower
\end{figureptx}%
\end{sectionptx}
\typeout{************************************************}
\typeout{Section 3 Direct Computation of the Abacus}
\typeout{************************************************}
\begin{sectionptx}{Direct Computation of the Abacus}{}{Direct Computation of the Abacus}{}{}{x:section:s-direct-computation}
 Next we describe how results can be computed much more easily. The rooted tree with root \(r\) and subtrees \(T_1\), \(T_2\), ... \(T_m\) is denoted \(\text{RootedTree}\left(r,T_1, T_2, \ldots ,T_m\right)\). \begin{definition}{Restoration Number.}{g:definition:idm354135319136}%
The restoration number of a rooted tree is number of times the root needs to fire to return the stochastic abacus to its critical loading position, denoted \(R(T)\).%
\end{definition}
The restoration number is also the number of chips in the root's absorbing vertex upon return to critical loading.%
\begin{definition}{Restoration Function.}{g:definition:idm354135317040}%
The restoration function of a rooted tree \(T\) is the function \(R_T\) on the vertex set of the tree such that \(R_T(v)\) is the number of chips in \(v\)'s absorbing vertex upon return to critical loading.%
\end{definition}
Note: The probability vertex \(v\) wins Pass the Buck on a rooted tree \(T\) is%
\begin{equation*}
\frac{R_T(v)}{\sum _{w\in V_T} R_T(w)}.
\end{equation*}
\begin{definition}{Period of a Rooted Tree.}{g:definition:idm354135313120}%
The Period of a rooted tree, \(T\), is the restoration number of the rooted tree having \(T\) as it's only subtree. The period of a tree \(T\) is denoted \(P(T)\).%
\end{definition}
Note: If \(T'= \text{RootedTree}(v;T)\), then \(R_{T'}(v)=P(T)\)%
\begin{theorem}{}{}{g:theorem:idm354135309440}%
If \(T\) has root \(v\) and has \(m\) subtrees with roots \(v_1, v_2, \ldots  v_m\), then%
\begin{equation}
P(T)=(m+2)R_T(v)-\sum_{i=1}^m R_T\left(v_i\right).\label{x:men:eq-0}
\end{equation}
\end{theorem}
\begin{proof}{}{g:proof:idm354135306688}
If we create a tree with root \(w\) having \(T\) as it's only subtree, the critical loading condition at  \(w\) occurs whenever \(T\) is critically loaded, in which case \(w\) can get sufficient chips without further firing. Therefore, we need to count how many chips, \(v\) needs. In order for \(v\) to fire, it needs to receive \(m+2\) chips each time, where \(m+2\) is the outdegree of \(v\) in the stochastic abacus.  The number that it needs,  \((m+2)R_T(v)\), does not all come from \(w\), however.   This number is decreased by one every time any of the roots of the subtrees of \(v\) fire, which accounts for the sum that is subtracted in \hyperref[x:men:eq-0]{({\xreffont\ref{x:men:eq-0}})}.\end{proof}
The significance of the period of a tree is that  when several subtrees combine with a root, the restoration number of the new rooted tree is a function of the periods of its children.%
\begin{theorem}{}{}{g:theorem:idm354135300224}%
Let \(T' = \text{RootedTree}\left(v,T_1, T_2, \ldots ,T_m\right)\), then%
\begin{equation}
R_{T'}(v) = \text{lcm}\left(P\left(T_1\right),
P\left(T_2\right), \ldots ,P\left(T_{m }\right)\right),\label{x:men:eq-1}
\end{equation}
and for each vertex \(w\) in \(T_k\),%
\begin{equation}
R_{T'}(w)= \frac{R_{T'}(v)}{P\left(T_k\right)}R_{T_k}(w)\text{.}\label{x:men:eq-2}
\end{equation}
\end{theorem}
\begin{proof}{}{g:proof:idm354135298752}
As the stochastic abacus is running, each subtree \(T_i\) reaches its own critical loading condition after \(D\left(T_i\right)\) root firings. Therefore, critical loading of all subtrees is first reached after the least common multiple of their periods, \hyperref[x:men:eq-1]{({\xreffont\ref{x:men:eq-1}})}. For each of the subtrees, the number of chips deposited in a period is multiplied by the number of periods that the subtree goes though, which accounts for \hyperref[x:men:eq-2]{({\xreffont\ref{x:men:eq-2}})}%
\end{proof}
This lets us determine the restoration function of any rooted tree from the bottom up. We illustrate the technique with the tree, \(\mathcal{T}\) in \hyperref[x:figure:fig-rooted-random]{Figure~{\xreffont\ref{x:figure:fig-rooted-random}}}.  Let \(\epsilon\left(v_i\right)\) be the trivial tree with a single vertex, \(v_i\), its root.  We know that \(R\left(\epsilon\left(v_i\right)\right)=1\), and \(P\left(\epsilon \left(v_i\right)\right) = 2\), \(i=1,2,3.\) Therefore the level 2 tree with three trivial subtrees, \(\tau_1=\text{RootedTree}\left(v;\epsilon \left(v_1\right),\epsilon \left(v_2\right),\epsilon \left(v_3\right)\right)\) has restoration value \(R\left(\tau_1\right)=2\). The period of \(\tau_1\) is \(P\left(\tau_1\right)= (3+2)R(v_4)-3 R_{\tau_1}(\epsilon ) = 7\). On the right side of the tree, we have the subtree \(\tau_2=\text{RootedTree}\left(v_6;\text{RootedTree}\left(v_5;\epsilon \left(v_4\right)\right)\right.\). We can determine the restoration function and period of this tree: \(R_{\tau_2}\left(v_4\right)=1\), \(R_{\tau_2}\left(v_5\right)=2\), \(R_{\tau_2}\left(v_6\right)= 5\), and \(P\left(\tau_2\right) = 13\).%
\par
Finally we can compute the restoration function of \(\mathcal{T}\):%
\begin{equation*}
R_{\mathcal{T}}\left(v_7\right)=\text{lcm}(7,13)=91.
\end{equation*}
We complete the computation of \(R_{\mathcal{T}}\) by multiplying \(R_{\tau_1}\) by 13 and \(R_{\tau_2}\) by 7. The final result agrees with the actual implementation of the stochastic abacus that was displayed in \hyperref[x:figure:fig-abacus-random-tree]{Figure~{\xreffont\ref{x:figure:fig-abacus-random-tree}}}.%
\end{sectionptx}
\typeout{************************************************}
\typeout{Section 4 Implementation of the Direct Calculation}
\typeout{************************************************}
\begin{sectionptx}{Implementation of the Direct Calculation}{}{Implementation of the Direct Calculation}{}{}{x:section:s-implementation}
In order to implement the process describe above, we use an array representation of rooted trees. We number the vertices in a tree with \(n\) vertices with the positive integers from \(1\) to \(n\).  The structure of the tree is encapsulated in an array of \(n\) integers, \mono{T}.%
\par
In general, the entry \mono{T[k]} contains the parent of vertex \mono{k}.  The root of the tree has no parent and if \mono{k} is the root, \mono{T[k]=0}. If we number the vertices in \hyperref[x:figure:fig-rooted-random]{Figure~{\xreffont\ref{x:figure:fig-rooted-random}}} by the subscripts of the vertex names, the tree would be represented by the array%
\begin{equation*}
(4, 4, 4, 8, 6, 7, 8, 0).
\end{equation*}
\par
The following Mathematica code will identify various parts of a rooted tree, assuming the structure we have described above.%
\par\medskip%
\noindent\phantomsection\label{x:fragment:frag-parts-of-tree}\textlangle{}1 \textrangle{} \(\equiv\)\index{frag-parts-of-tree@frag-parts-of-tree}\\
\begin{preformatted}
root[T_] := FirstPosition[T, 0] // First
children[T_, k_] := Position[T, k] // Flatten
leafQ[T_, j_] := Not[MemberQ[T, j]]
descendants[T_, j_] := {} /; leafQ[T, j]
descendants[T_, j_] := 
 Join[children[T, j], 
      Join @@ Map[descendants[T, #]&, children[T, j]]] /; Not[leafQ[T, j]]
\end{preformatted}
The following functions converts a tree in the form of undirected edges with designated root into the array form we use in our implementation.%
\par\medskip%
\noindent\phantomsection\label{x:fragment:frag-conversion}\textlangle{}2 \textrangle{} \(\equiv\)\index{frag-conversion@frag-conversion}\\
\begin{preformatted}
treeArray[el_List, root_] := 
 	If[AcyclicGraphQ[Graph[el]] && ConnectedGraphQ[Graph[el]], 
 		 maketreeArray[el, root], "error"]

maketreeArray[el_List, root_] := 
 Module[{ta, n}, n = Length[el] + 1; ta = Table[0, {n}]; 
  Map[FindShortestPath[el, root, #]&, Complement[Range[n], {root}]] //
           Map[Partition[#, 2, 1]&, #]& // 
           Flatten[#, 1]& // 
           Union //
           Map[(ta[[#[[2]]]] = #[[1]])&, #]&; 
	ta]
\end{preformatted}
This function computes the restoration function a tree in the form of a list of undirected edges with designated root. The expression \mono{r[T,k,j]} represents the restoration function of the subtree within \mono{T} rooted at \mono{k} evaluated for the vertex \mono{j}; and \mono{p[T,k]} is the period of the subtree of \mono{T} rooted at \mono{k}.%
\par\medskip%
\noindent\phantomsection\label{x:fragment:frag-abacus-calcs}\textlangle{}3 \textrangle{} \(\equiv\)\index{frag-abacus-calcs@frag-abacus-calcs}\\
\begin{preformatted}
restoration[tree_, root_] := 
 Module[{r, p, ta, n}, 
 	ta = treeArray[tree, root]; 
  	n = Length[tree] + 1; r[T_, k_, k_] := 1 /; leafQ[T, k];
  	p[T_, k_] := 2 /; leafQ[T, k];
  	r[T_, k_, k_] := 
   	LCM @@ Map[p[T, #]&, children[T, k]] /; Not[leafQ[T, k]];
  	p[T_, k_] := 
        p[T, k] = (2 + Length[children[T, k]]) r[T, k, k] - 
                     Total[Map[r[T, k, #]&, children[T, k]]];
  	r[T_, k_, j_] := 
   	r[T, k, j] = 
    		Module[{i}, 
     				i = (Select[children[T, k], 
         			MemberQ[Join[{#}, descendants[T, #]], j]&]) // First; 
     				r[T, k, k] r[T, i, j]/p[T, i]]; 
  	Map[{#, r[ta, root, #]}&, Range[n]] ]
\end{preformatted}
Here, we test the code with the example of \hyperref[x:figure:fig-rooted-random]{Figure~{\xreffont\ref{x:figure:fig-rooted-random}}} and see that it is consistant with the output of the abacus seen in \hyperref[x:figure:fig-abacus-random-tree]{Figure~{\xreffont\ref{x:figure:fig-abacus-random-tree}}}.%
\par\medskip%
\noindent\phantomsection\label{x:fragment:frag-test1}\textlangle{}4 \textrangle{} \(\equiv\)\index{frag-test1@frag-test1}\\
\begin{preformatted}
test1 = {UndirectedEdge[8, 4], UndirectedEdge[8, 7], UndirectedEdge[4, 1],
         UndirectedEdge[4, 2], UndirectedEdge[4, 3], UndirectedEdge[7, 6], 
         UndirectedEdge[6, 5]};
restoration[test1,8]

{{1,13},{2,13},{3,13},{4,26},{5,7},{6,14},{7,35},{8,91}}
\end{preformatted}
The tree in \hyperref[x:figure:fig-larger-tree]{Figure~{\xreffont\ref{x:figure:fig-larger-tree}}} is more complex, yet also not huge, but its restoration number is considerably larger than the previous example.%
\par\medskip%
\noindent\phantomsection\label{x:fragment:frag-test2}\textlangle{}5 \textrangle{} \(\equiv\)\index{frag-test2@frag-test2}\\
\begin{preformatted}
test2 =Map[UndirectedEdge@@#&, 
	{{26, 22}, {25, 24}, {25, 23}, {25, 22}, {22, 15}, {21, 17}, 
 	{20, 19}, {20, 18}, {20, 17}, {17, 16}, {16, 15}, {15, 1},  
 	{14, 10}, {13, 12}, {13, 11}, {13, 10}, {10, 3}, {9, 5}, {8, 7}, 
 	{8, 6}, {8, 5}, {5, 4}, {4, 3}, {3, 2}, {2, 1}}]
restoration[test2,1]

{{1, 27783522}, {2, 10297681}, {3, 3109521}, {4, 1158449}, {5, 365826}, 
{6, 60971}, {7, 60971}, {8, 121942}, {9, 182913}, {10, 981954},
{11, 163659}, {12, 163659}, {13, 327318}, {14, 490977}, {15, 8389602},
{16, 3125538}, {17, 987012}, {18, 164502}, {19, 164502}, {20, 329004},
{21, 493506}, {22, 2649348}, {23, 441558}, {24, 441558},
{25, 883116}, {26, 1324674}}
\end{preformatted}
\end{sectionptx}
\typeout{************************************************}
\typeout{References  References}
\typeout{************************************************}
\begin{references-section-numberless}{References}{}{References}{}{}{g:references:idm354135260816}
\begin{referencelist}
\bibitem[1]{x:biblio:biblio-bjoner}\hypertarget{x:biblio:biblio-bjoner}{}Bjöner, A., Lovasz, L., Shor, P. (1991), \textit{Chip-firing games on graphs}, Eur. J. Combin. \textbf{12 (4)}, 283\textendash{}291, doi.org\slash{}10.1016\slash{}s0195-6698(13)80111-4.
\bibitem[2]{x:biblio:biblio-engel}\hypertarget{x:biblio:biblio-engel}{}Arthur Engel (1976), \textit{Why does the probabilistic abacus work?}, Educational Studies in Mathematics \textbf{7}, 59\textendash{}69.
\bibitem[3]{x:biblio:biblio-kemeny}\hypertarget{x:biblio:biblio-kemeny}{}John G. Kemeny and J. Laurie Snell, \textit{Finite Markov Chains}, Undergraduate Texts in Mathematics, Springer- Verlag, New York, 1976.
\bibitem[4]{x:biblio:biblio-levasseur}\hypertarget{x:biblio:biblio-levasseur}{}Levasseur, K. (2021), \textit{Pass the Buck on a Complete Binary Tree}, Mathematics Magazine, to appear.
\bibitem[5]{x:biblio:biblio-propp}\hypertarget{x:biblio:biblio-propp}{}Propp, J. (2018), \textit{Prof.  Engel’s  marvelously  improbable  machines}, Math Horizons, 26(2):  5–9. doi.org\slash{}10.1080\slash{}10724117.2018.1518840.
\bibitem[6]{x:biblio:biblio-snell}\hypertarget{x:biblio:biblio-snell}{}J. Laurie Snell, \textit{The Engel algorithm for absorbing Markov chains}, Available at https:\slash{}\slash{}arxiv.org\slash{}abs\slash{}0904.1413v1
\bibitem[7]{x:biblio:biblio-torrence}\hypertarget{x:biblio:biblio-torrence}{}Bruce Torrence, \textit{Passing the Buck and Firing Fibonacci: Adventures with the Stochastic Abacus}, The American Mathematical Monthly, May 2019, \textbf{126} no.\@\,5, 387\textendash{}399, doi.org\slash{}10.1080\slash{}00029890.2019.1577089.
\end{referencelist}
\end{references-section-numberless}
\end{document}